\documentclass[12pt,twoside]{article}
\usepackage{amsmath}
\usepackage{amsthm}
\usepackage{amssymb}
\usepackage[margin=1in]{geometry}
\usepackage{graphicx}
\usepackage{ifpdf}
\usepackage{color}

\ifpdf
 \DeclareGraphicsExtensions{.pdf}
\else
 \DeclareGraphicsExtensions{.eps}
\fi

\title{A logarithmic estimate for the inverse source scattering problem with attenuation in a two-layered medium}
\author{ Mozhgan Nora Entekhabi \quad \quad Ajith Gunaratne }
\date{ }
\begin{document}
\maketitle
\newtheorem{theorem}{Theorem}[section]
\newtheorem{lemma}[theorem]{Lemma}
\newtheorem{corollary}[theorem]{Corollary}
\newtheorem{definition}[theorem]{Definition}
\newtheorem{proposition}[theorem]{Proposition}
\[
Department\; of\; Mathematics
\]
\[
Florida\; A\; \& \;M \; University
\]
\[
Tallahassee, \; Fl \; 32307,\; U.S.A.
\]
\[
e-mails:\;\; mozhgan.entekhabi@famu.edu, \quad  ajith.gunaratne@famu.edu
\]

{\bf Abstract.} The paper aims a logarithmic stability estimate for the inverse source problem of the one-dimensional Helmholtz equation with attenuation factor in a two layer medium. We establish a stability by using multiple frequencies at the two end points of the domain which contains the compact support of the source functions.  

\vspace{10pt}

\textbf{Keywords:} Inverse source problems, scattering theory, exact observability. 

\vspace{10pt}

\textbf{Mathematics Subject Classification(2000)}: 35R30; 35J05; 35B60; 33C10; 31A15; 76Q05; 78A46

\section{Introduction and problem formulation}
Inverse source problem arises in many areas of science. It has numerous applications to surface vibrations, acoustical and biomedical/medical imaging, antenna synthesis, geophysics, and material science (\cite{ABF, B}). It has been known that the data of the inverse source problems for Helmholtz equations with single frequency can not guarantee the uniqueness (\cite{I}, Ch.4). On the other hand, various studies, for instance in \cite{EV}, showed that the uniqueness can be regained by taking multi-frequency boundary measurement in a non-empty frequency interval $(0,K)$ noticing the analyticity of wave-field on the frequency. Because of the wide applications, these problems have attracted considerable attention. For example, In the papers \cite{CIL, EI}  sharp results were obtained in sub-domains of $\mathbb{R}^3$ and $\mathbb{R}^2$  with a possibility of handling spatially variable coefficients. An iterative/recursive algorithm was developed for recovering unknown sources in \cite{ BLT, BLT1,BLRX}. In papers \cite{ E, IL}, authors considered Helmholtz equation with damping factor. Authors in \cite{LY}, improved the stability for the source when \textcolor{red}{the} domain is a disk/ball. In \cite{J2} the uniform logarithmic stability with respect to the wave numbers for continuation of the Helmholtz equation from the unit disk onto any larger disk was studied and recently \cite{AI, IK, IL1} showed the increasing stability for continuation problems with large wave number under (pseudo) convexity conditions on the domain. We also have to mention that in \cite{EI2} inverse source \textcolor{red}{problem} was considered for classical elasticity system.\\

In particular attenuation can have various reasons and in application one of \textcolor{red}{the} fundamental reasons of poor resolution in inverse problems is a spatial decay of the signal due in part to the damping factor. The main purpose of this paper is to study the dependence of increasing stability on the constant attenuation (damping) coefficient in the inverse scattering source problems. Our result in agreement with the result of paper \cite{YL}, if the damping factor \textcolor{red}{becomes
} zero. To achieve our goal we used analytic continuation, Carleman estimates for damped wave equation and exact observability bounds for hyperbolic equations which was recently developed in \cite{CIL}. In this paper, we assume that the medium is homogeneous in the whole space. Here we try to establish \textcolor{red}{a stability estimate to
recover of the source functions for the inverse source problem for the one-dimensional Helmholtz equation in a two-layered medium with attenuation factor $\alpha$. In this paper, the damping factor is considered the same for both layers of medium. }\\

In this paper both functions $f_0 \in H^2((-1,1)), f_1 \in H^1((-1,1))$ are assumed to be zero outside our domain and $suppf_0 \cup suppf_1 \subset (-1,1) $. In this work for simplicity we used $\partial \Omega$ instead of our boundary which is $\{1,-1\}$.  We consider the following attenuated Helmholtz equation in a two-layered medium 

\begin{equation}
\label{PDE}
u'' + (k^2(x)+i\alpha k
(x)) u=-f_1 -\alpha f_0  +ikf_0, \quad x\in(-1,1),
\end{equation}
with the exponential decay at infinity condition 
\begin{equation}
\label{outgoing cond}
    |u(x)|+|u'(x)|\leq C(u)e^{\delta (u)|x|},
\end{equation}
where $\delta (u)>0$ and wave number $k$ defines as follows

\begin{equation}
\label{Wavenumber2}
    k(x)= \left\{
	\begin{array}{ll}
	k_p  & \mbox{if } x > 0 \\
	k_n & \mbox{if } x < 0,
	\end{array}
\right.
\end{equation}

Our goals are uniqueness and stability of \textcolor{red}{the functions} $ f_0, f_1 $ from the Dirichlet data. Now let  
\begin{equation}
\label{u*}
u^*(x,\kappa):=u(x,k), \quad \kappa^2 := k^2+ \alpha ki,
\end{equation}
then the equation \eqref{PDE} becomes 
\begin{equation}
\label{u*PROBLEM}
u^{*''}+\kappa ^2 u^*  =-f_1 - \alpha f_0 + i kf_0,
\end{equation}
and also we can reformulate (3) as follows
\begin{equation}
\label{Wavenumber}
    \kappa(x)= \left\{
	\begin{array}{ll}
	\kappa_p=(k_p ^2+ \alpha k_pi)^{1/2}   \\
	\kappa_n =(k_n ^2+ \alpha k_ni)^{1/2} .
	\end{array}
\right.
\end{equation}
Using standard formulation of Helmholtz equation, $\kappa_p=c_p\omega, \kappa_n=c_n\omega$, $\omega>0$ is the angular frequency and $c_p, c_n $ are constants.\\ It is well-known that equation \eqref{u*PROBLEM} has \textcolor{red}{the} following unique solution provided $f_1 \in L^2(-1,1) , f_0 \in H^1(-1,1)$: 

\begin{equation}
 u^*(x,\omega)=\int_0^1 G(x,y) ( -f_1 -\alpha f_0+ikf_0) (y)dy,
\end{equation}
where $G(x,y)$ is the Green function given as follows

\begin{equation*}
\label{GreenFunc}
   G(x,y)= \left\{
	\begin{array}{ll}
i\frac{\kappa_p-\kappa_n}{2\kappa_p(\kappa_p+\kappa_n)} e^{i\kappa_p(x+y)}+ \frac{i}{2\kappa_p}e^{i\kappa_p|x-y|,} & \mbox{if } x > 0, \\
	\frac{i}{(\kappa_p+\kappa_n)} e^{i(\kappa_py-\kappa_nx)}, & \mbox{if } x < 0,
	\end{array}
\right.
\mbox{ for }\quad   y>0,
\end{equation*}

and 
\begin{equation*}
   G(x,y)= \left\{
	\begin{array}{ll}
i\frac{\kappa_p-\kappa_n}{2\kappa_p(\kappa_p+\kappa_n)} e^{-i\kappa_p(x+y)}+ \frac{i}{2\kappa_p}e^{i\kappa_p |x-y|}, & \mbox{if } x > 0, \\
	\frac{i}{(\kappa_p+k_n)} e^{i(-\kappa_ny+\kappa_px)}, & \mbox{if } x < 0,
	\end{array}
\right.
\mbox{ for }\quad   y<0.
\end{equation*}

To see more detail for the Green function see \cite{YL}. The following logarithmic estimate states our main result.
\begin{theorem}
\label{maintheorem}
There exists a generic constant $C$ depending on the domain $(-1,1) $ such that 

\begin{equation}
\label{Istability}
\parallel f_0 \parallel_{(0)}^{2}(-1,1)+ \parallel f_1 \parallel_{(0)}^{2}(-1,1) \leq      Ce^{C\alpha ^2}\Big(\epsilon ^2+\frac{(\alpha ^2+1)M^{2}}{K^{\frac{2}{3}}E^{\frac{1}{4}}+1} \Big), 
\end{equation}
for all $u\in H^2 ((-1,1))$ solving \eqref{PDE} with $K>1$. Here 

\begin{equation*}
\label{epsilon}
\epsilon ^2  = \int _{0} ^{K} \omega ^2\big(  | u(1,\omega) |^{2} + | u(-1,\omega) |^{2} \big ) d\omega,
\end{equation*}
$E=-ln\epsilon $ and $M = \max \big \{ \parallel f_0 \parallel_{(2)}(-1,1)+ \parallel f_1 \parallel_{(1)}(-1,1), 1 \big\}$ where $\parallel . \parallel _{(l)}((-1,1))$ is the standard Sobolev norm in $H^l((-1,1))$.
\end{theorem}
\textbf{Remark 1.1.} Estimate \eqref{Istability} implies that for any fixed $ \alpha $, \textcolor{red}{the} stability of $f_1, f_0$ from the boundary data is improving with growing $k$. \textcolor{red}{In another words}, the problem becomes more stable when higher frequency data \textcolor{red}{is used}, but it also implies that larger attenuation will deteriorate this improvement. The right hand-side of  \textcolor{red}{the} estimate \eqref{Istability} consists of two parts: data discrepancy and the high frequency tail. There \textcolor{red}{is} some numerical evidence that when $K$ grows,  \textcolor{red}{the} functions $f_0, f_1$ will have better resolution. In addition, our estimate is a proof for the uniqueness of the inverse source functions as $\epsilon \rightarrow 0$.

\section{Increasing Stability of Continuation to higher frequencies}
To proof our main  \textcolor{red}{theorem}, let's define the following  functions:

\begin{center}

\begin{equation*}
f_{1p} = \begin{cases}
            	f_1  & \mbox{if } x > 0, \\
0 & \mbox{if } x < 0,
       \end{cases} \quad
f_{1n} = \begin{cases}
            0  & \text{if } x > 0, \\
            f_1 & \text{if } x < 0,
       \end{cases}
\end{equation*}
\begin{equation*}
f_{0p} = \begin{cases}
            	f_0  & \mbox{if } x > 0, \\
	0& \mbox{if } x < 0,
       \end{cases} \quad
f_{0n} = \begin{cases}
             0  & \text{if } x > 0, \\
          f_0  & \text{if } x < 0.
       \end{cases}
\end{equation*}
 \end{center}
\label{S:1}
And also defining  

\begin{equation*}
    I(k)=I_1(k)+I_2(k)
\end{equation*}
where 
\small{
\begin{equation*}
\label{I1-I2}
I_1(k)= \int _{0} ^{k} \omega ^2 | u(-1,\omega) |^{2} d\omega, \quad  I_2 (k)= \int _{0}^{k}  \omega ^2 | u(1,\omega) |^{2} d\omega,
\end{equation*}
using \eqref{u*}, we can show that 
\begin{equation}
\label{omegau-1}
    \omega u(-1,\omega)=\int_{0}^{1}\frac{i}{c_1 + c_2}e^{i(c_1\omega y+c_2 \omega)} ( -f_{1p} -\alpha f_{0p}+ikf_{0p}) (y)dy 
\end{equation}
\begin{equation*}
   + \int_{-1}^{0}\frac{i(c_2-c_1)}{2c_2(c_1 + c_2)}e^{-ic_2\omega (-1+y)} ( -f_{1n} -\alpha f_{0n}+ikf_{0n}) (y)dy 
\end{equation*}

\begin{equation*}
   + \int_{-1}^{0}\frac{i}{2c_2}e^{-ic_2\omega (-1-y)} ( -f_{1n} -\alpha f_{0n}+ikf_{0n}) (y)dy
\end{equation*}
and 

\begin{equation}
\label{omegau+1}
    \omega u(1,\omega)=\int_{0}^{1}\frac{i(c_1-c_2)}{2c_1(c_1 +c_2)}e^{ic_1 \omega(1+y)} ( -f_{1p} -\alpha f_{0p}+ikf_{0p}) (y)dy 
\end{equation}
\begin{equation*}
   + \int_{0}^{1}\frac{i}{2c_1 }e^{ic_2\omega (1-y)} ( -f_{1p} -\alpha f_{0p}+ikf_{0p}) (y)dy 
\end{equation*}

\begin{equation*}
   + \int_{-1}^{0}\frac{i}{c_1+c_2}e^{i(-c_2\omega y+ c_1 \omega)} ( -f_{1n} -\alpha f_{0n}+ikf_{0n}) (y)dy.
\end{equation*}

}
Functions $I_1$ and $I_2$ are both analytic with respect to the wave number $k \in \mathbb{C}\backslash [0, -i\alpha]  $ and play important roles in relating the inverse source problems of the Helmholtz equation and the Cauchy problems for the wave equations. \\
\begin{lemma}
Let $supp f_0 , supp f_1 \in (-1,1)$ and $f_0 \in H^0 (-1,1) $, $f_1 \in H^0 (-1,1) $. Then 
\begin{equation}
\label{boundI1}
|I_1(k)|\leq C\Big( |k| \parallel f_1 \parallel_{(0)}^2 (-1,1) +( |k|\alpha ^2 + \frac{|k|^3}{3})\parallel f_0 \parallel_{(0)}^2 (-1,1) \Big)e^{4c_{max}\big (4k_{1} + \alpha\big )},
\end{equation}
\begin{equation}
\label{boundI2}
|I_2(k)|\leq C\Big( |k| \parallel f_1 \parallel_{(0)}^2 (-1,1) +( |k|\alpha^2 + \frac{|k|^3}{3})\parallel f_0 \parallel_{(0)}^2 (-1,1) \Big)e^{4c_{max}\big (4k_{1} + \alpha \big )},
\end{equation}

\end{lemma}

\begin{proof}
Since we have $\kappa = \kappa _1 +\kappa _2 i= \sqrt[]{k^2 +\alpha ki}$ is complex analytic on $\mathbb{C}\setminus [0,-\alpha i]$ and in particular on the set $\mathbb{S}\setminus [0,k]$, where $\mathbb{S}$ is the sector $\{|\arg k|<\pi/4 \}$ with $k=k_1 +ik_2$. It is easy to see that $|\kappa|= |k|^{\frac{1}{2}}|k+ \alpha i|^{\frac{1}{2}} \leq  2k_{1} ^{1/2}(k_{1} ^{1/2} +\alpha )^{1/2}$ and $|k|\leq \sqrt[]{2}k_1  \leq  \sqrt[]{2} |\kappa| $ for any $k$ in $\mathbb{S}$. By a simple calculation, we can show that
\begin{equation}
\label{I1}
I_1 (k)= \int _{0}^{k}\big|\int_{0}^{1}\frac{1}{c_1 + c_2}e^{i(c_1\omega y+c_2 \omega)} ( -f_{1p} -\alpha f_{0p}+ikf_{0p}) (y)dy 
\end{equation}

\begin{equation*}
   + \int_{-1}^{0}\frac{(c_2-c_1)}{2c_2(c_1 + c_2)}e^{-ic_2\omega (1+y)} ( -f_{1n} - \alpha f_{0n}+ikf_{0n}) (y)dy 
\end{equation*}

\begin{equation*}
   + \int_{-1}^{0}\frac{1}{2c_2}e^{-ic_2\omega (-1-y)} ( -f_{1n} -\alpha f_{0n}+kf_{0n}) (y)dy\big|^2d\omega
\end{equation*}
and 
\begin{equation}
\label{I2}
    I_2 (k)=\int _{0}^{k}\big| \int_{0}^{1}\frac{(c_1-c_2)}{2c_1(c_1 -c_2)}e^{i(c_1\omega y+c_2 \omega)} ( -f_{1p} -\alpha f_{0p}+kf_{0p}) (y)dy 
\end{equation}
\begin{equation*}
   + \int_{0}^{1}\frac{1}{2c_1 }e^{-ic_2\omega (1-y)} ( -f_{1p} -\alpha f_{0p}+kf_{0p}) (y)dy 
\end{equation*}

\begin{equation*}
   + \int_{-1}^{0}\frac{1}{c_1+c_2}e^{i(-c_2\omega y+ c_1 \omega)} ( -f_{1n} -\alpha f_{0n}+kf_{0n}) (y)dy\big|^2 d\omega.
\end{equation*}
Since the integrands in \eqref{I1} and \eqref{I2} are analytic functions of $k$ in $ \mathbb{S}$, their integrals with respect to $\omega$ can be taken over any path in $\mathbb{S}$ joining points $0$ and $k$ in the complex plane. \textcolor{red}{Using} the change of variable $\omega=ks$, $s\in (0,1)$ in the line integral \eqref{u*} we obtain
\begin{equation*}
 |I_1 (k) |\leq  \int_{0}^{1}|k| \big|\int_{0}^{1}\frac{1}{c_1 + c_2}e^{i(c_1\omega y+c_2 \omega)} ( -f_{1p} -\alpha f_{0p}+ikf_{0p}) (y)dy 
 \end{equation*}

\begin{equation*}
   + \int_{-1}^{0}\frac{(c_2-c_1)}{2c_2(c_1 + c_2)}e^{-ic_2\omega (1+y)} ( -f_{1n} - \alpha f_{0n}+ikf_{0n}) (y)dy 
\end{equation*}

\begin{equation*}
   + \int_{-1}^{0}\frac{1}{2c_2}e^{-ic_2\omega (-1-y)} ( -f_{1n} -\alpha f_{0n}+kf_{0n}) (y)dy\big|^2d\omega,
\end{equation*}
using the following inequalities for $y\in(-1,1)$
\begin{equation*}
    |e^{\pm i\omega(c_1y + c_2)}|\leq e^{2c_{max} |\kappa_2|}, \quad  |e^{\pm ic_2 \omega(\pm y-1 )}|\leq e^{2c_{max} |\kappa_2|},
\end{equation*}
it is  easy to drive that
\begin{equation*}
 |I_1 (k) |\leq  \int_{0}^{1}|k| \int _{-1 }^{1}  \Big(|f_{1} (y)|+(\alpha+ s|k|)|f_{o} (y)|  e^{2c_{max} |\kappa_2|} dy \Big)^2  ds,
 \end{equation*}
integrating with respect to $s$, using the bound for $|\kappa|$ in $\mathbb{S}$ and trivial inequality $2ab\leq a^2+b^2$, we complete the proof of \eqref{boundI1}.\\ 
Similarly for $y\in(-1,1)$ we have
\begin{equation*}
    |e^{\pm i\omega c_2(\pm y-1 )}|\leq e^{2c_{max} |\kappa_2|}, \quad  |e^{\pm i\omega(c_1 - c_2y )}|\leq e^{2c_{max} |\kappa_2|},
\end{equation*}
using the same technique for $I_2 (k)$, proof for \eqref{boundI2} is complete. 

\end{proof}
The following steps are essential to link the unknown $I_1(k)$ and $I_2(k)$ for $k\in [K,\infty)$ to the known value $\epsilon$ in \eqref{PDE}. Obviously

\begin{equation}
\label{Ie^}
|I_1(k) e^{-16c_{max} k}|\leq 
\end{equation}
\begin{equation*}
C\Big( |k| \parallel f_1 \parallel_{(0)}^2 (-1,1) +( |k|\alpha ^2 + \frac{|k|^3}{3})\parallel f_0 \parallel_{(0)}^2 (-1,1) \Big)e^{4c_{max} \alpha},
\end{equation*}
\begin{equation*}
\leq C\alpha ^2 e^{4c_{max}\alpha}M^{2},
\end{equation*}
with $M=\max \big \{  \parallel f_1 \parallel_{(0)}^2  (-1,1) +\parallel f_0 \parallel_{(0)}^2  (-1,1) , 1 \big \}$. With the similar argument bound \eqref{Ie^} is true for $I_2 (k)$. Observing that 

\begin{equation*}
|I_1(k)  e^{-16c_{max} k}|\leq \epsilon ^2, \quad |I_2(k)  e^{-16c_{max} k}|\leq \epsilon ^2  \textit{ on } [0, K].
\end{equation*}
Let $\mu (k) $ be the harmonic measure of the interval $[0,K]$ in $\mathbb{S}\backslash [0,K], $ then as known (for example see \cite{I}, p.67), from two previous inequalities and analyticity of the function $I_{1}(k) e^{-16c_{max} k}$ and $ I_{2}(k) e^{-16c_{max} k} $ we conclude that 
\begin{equation}
\label{I1Epsilon}
|I_1(k) e^{-16c_{max} k}|\leq  C\alpha^2 e^{4c_{max}\alpha}\epsilon ^{2 \mu(k)} M^{2},
\end{equation}
when $K<k< +\infty$. Similar arguments also yield for
\begin{equation}
\label{I2Epsilon}
|I_2(k) e^{-16c_{max} k}|\leq  C\alpha^2 e^{4c_{max}\alpha}\epsilon ^{2 \mu(k)} M^{2},
\end{equation}
hence
\begin{equation}
\label{I2Epsilon}
|I(k) e^{-16c_{max} k}|\leq  C\alpha^2 e^{4c_{max}\alpha}\epsilon ^{2 \mu(k)} M^{2}.
\end{equation}

To achieve a lower bound of the harmonic measure $\mu (k)$, we use the following technical lemma. The proof can be found in \cite{CIL}.
\begin{lemma}
If $0<k<2^{\frac{1}{4}}K$, then 
\begin{equation}
\frac{1}{2}\leq \mu(k).
\end{equation}
If $2^{\frac{1}{4}}K <k$, then
\begin{equation}
\label{harmonicbound}
\frac{1}{\pi} \Big( \big ( \frac{k}{K} \big)^{4} -1 \Big)^{\frac{-1}{2}} \leq \mu (k).
\end{equation}
\end{lemma}
\vspace{0.5cm}
\subsubsection{Exact observability bound for wave equation with damping factor}
In order to use the bound for higher frequency, we consider the hyperbolic initial value problem
\begin{equation}
\partial^{2}_{t}U(x,t) - U''(x,t)+\alpha \partial_t U(x,t)=0 \textit{ on } (-1,1) \times (0,\infty), \, U(x,0)= f_0, \, \partial_t U(x,0)=f_1 \textit{ on } (-1,1).
\end{equation}
By assuming $\alpha =0$, the exact observability bounds for the hyperbolic equation can be found in \cite{CIL,EI}. The following theorem presents a generalized result which is of its own interest.

\begin{theorem}
Let the observation time $4(D+1) <T< 5(D+1)$. Then there exists a generic constant $C$ depending on the domain $\Omega$ such that
\begin{equation}
\parallel f_0\parallel_{(1)}^{2}(\Omega)+\parallel f_1\parallel_{(0)}^{2} (\Omega) \leq C e^{C\alpha^2} \Big( \parallel \partial_t U \parallel_{(0)}^{2}(\partial \Omega \times(0,T))+ \parallel  U \parallel_{(0)}^{2}(\partial \Omega \times(0,T)) \Big),
\end{equation}
for all $U \in H^2((-1,1)\times (0,\infty))$ solving (3.1). 
\begin{proof}
Proof is in (\cite{IL}, Lemma 3.3 and Theorem 3.1).
\end{proof}
\end{theorem}

\section{Logarithmic stability for inverse source problem}
To proceed, we consider the hyperbolic initial value problem
\begin{equation}
\label{U0}
\partial^{2}_{t}U - U''+\alpha \partial_t U=0 \textit{ on }\mathbb{R} \times (0,\infty), \, U(x,0)= f_0, \, \partial_t U(x,0)=f_1 \textit{ on } \mathbb{R}
\end{equation}
Defining $U(x,t) =0$ for $t<0$. We claim that the solution of \eqref{PDE} coincides with the Fourier transform of $U$;

\begin{equation}
\label{uF}
u(x,k)= \int_{-\infty}^{\infty} U(x,t)e^{ikt} dt.
\end{equation}
Known results in \cite{CFZ} (see Theorem 1.1. and Theorem 1.2.), \cite{KNO} and the assumption on the functions $f_0, f_1$ imply that 
\begin{equation*}
    \parallel U(.,t)\parallel_{(0)} \leq C(f_0,f_1)(1+t)^{-\frac{1}{4}},
\end{equation*}
so the Fourier transform \eqref{uF} is well defined. To prove the claim, the following steps are essential. Defining function $u_* (x,k)$ as the right hand side of \eqref{uF};
\begin{equation}
u_*(x,k)= \frac{1}{\sqrt[]{\pi}}\int _{|x|-1}^{\infty} U(x,t)e^{ikt}dt.
\end{equation}
We observe that due to speed of the propagation, $U(x,t)=0$ when $1+tc_{max} <|x|$, $x\in (-1,1)$ (see \cite{J}).
Using integration by parts and \eqref{U0}, it is easy to see that
\begin{equation*}
0= \int ^{\infty}_{0} \big(\partial^{2}_{t}U(.,t) - U(.,t)''+\alpha \partial_t U(.,t)\big ) e^{ikt}dt
\end{equation*}
\begin{equation*}
= -\partial_t  U(.,0) - \int_{0}^{\infty} ik\partial_t  U(.,t)e^{ikt} dt- \int _{0}^{\infty}  U(.,t)''e^{ikt}dt-\alpha U(.,0)
\end{equation*}
\begin{equation*}
-\alpha ik\int^{0}_{\infty} U(.,t)e^{ikt}dt
\end{equation*}
\begin{equation*}
=-\partial_t  U(.,0)-\alpha U(.,0)+ikU(.,0)-\int_{0}^{\infty} \big ( k^2 U(.,t)+ U(.,t)''+i \alpha kU(.,t) \big)e^{ikt}dt.
\end{equation*}
All above holds when $k=k_1+ik_2$, $k_2>0$. Considering the well known integral representation of the solution $U$ when $k_2 >0$ for large $|x|$ as in \cite{CH}, p. 695 and \cite{KNO}, $U$ decays for large $t$ (Bessel functions of first kind are bounded functions). 

 Hence above bound for $U$ combined with the exponential decay of $e^{ikt}$ with respect to $t$ for $k_2>0$ implies an exponential decay of $u_*(x,k)$ when $|x|$ is getting large. By standard argument from stationary scattering theory, function $u(x,k)$ decays exponentially. Hence we conclude that $u(x,k)=u_*(x,k)$ or \eqref{uF} for all $k_2>0$. Due to the $L^2$-continuity of both side with respect to $k_2$ we conclude \eqref{uF} for $k_2=0$.\\

To proceed the estimate for reminders of the whole integrands in \eqref{I1} and \eqref{I2} for $(k, \infty)$, we need the following lemma.
\begin{lemma}
\label{BoundK-Infty}
Let $u$ be a solution to the forward problem \eqref{PDE} with $f_1 \in H^1((-1,1))$ and $f_0 \in H^2((-1,1)))$ with $supp f_0, supp f_1\subset (-1,1)$, then 
\begin{equation}
\label{lemma4.1}
\int_{k}^{\infty} \omega ^2 | u(-1,\omega) |^2 d\omega+ \int_{k}^{\infty}   \omega ^2|u(1,\omega) |^{2} d\omega 
 \end{equation}
\begin{equation*}
\label{uomega}
 \leq C k ^{-1}\Big( (1+\alpha^2) \parallel f_0 \parallel ^2  _{(2)} (-1,1)  +\parallel f_1 \parallel ^2  _{(1)} (-1,1)  \Big),
\end{equation*}

\end{lemma}

\begin{proof}
Using \eqref{omegau-1} and \eqref{omegau+1}, we obtain
\begin{equation}
\label{Line1lemma4.1}
   \int _{k}^{\infty} \omega ^2 | u(-1,\omega) |^{2}d\omega + \int _{k}^{\infty} \omega ^2 | u(1,\omega) |^{2}d\omega
\end{equation}
\begin{equation*}
    \leq  \int _{k}^{\infty}\Big | \int _{0}^{1} e^{ic_1 \omega y} ( -f_{1p} -\alpha f_{0p}+kf_{0p}) (y)dy  \Big|^2d\omega + \int _{k}^{\infty}\Big | \int _{0}^{1} e^{-ic_1 \omega y} ( -f_{1p} -\alpha f_{0p}+k f_{0p}) (y)dy  \Big|^2d\omega 
\end{equation*}
\begin{equation*}
  +  \int _{k}^{\infty}\Big | \int _{-1}^{0} e^{ic_2 \omega y} ( -f_{1n} -\alpha f_{0n}+kf_{0n}) (y)dy  \Big|^2d\omega + \int _{k}^{\infty}\Big | \int _{-1}^{0} e^{-ic_2 \omega y} ( -f_{1n} -\alpha f_{0n}+kf_{0n}) (y)dy,  \Big|^2d\omega 
\end{equation*}

\begin{equation}
\label{Line2lemma4.1}
    \leq \int _{k}^{\infty}\Big | \int _{0}^{1} e^{ic_1 \omega y}  f_{1p} (y)dy  \Big|^2d\omega + \int _{k}^{\infty}\Big | \int _{0}^{1} (\alpha +k)e^{ic_1 \omega y}f_{0p} (y)dy  \Big|^2d\omega 
\end{equation}
\begin{equation*}
   + \int _{k}^{\infty}\Big | \int _{0}^{1} e^{-ic_1 \omega y}  f_{1p} (y)dy  \Big|^2d\omega + \int _{k}^{\infty}\Big | \int _{0}^{1} (\alpha +k)e^{-ic_1 \omega y}f_{0p} (y)dy  \Big|^2d\omega 
\end{equation*}
\begin{equation*}
    + \int _{k}^{\infty}\Big | \int _{-1}^{0} e^{ic_2 \omega y}  f_{1n} (y)dy  \Big|^2d\omega + \int _{k}^{\infty}\Big | \int _{-1}^{0} (\alpha +k)e^{ic_2 \omega y}f_{0n} (y)dy  \Big|^2d\omega 
\end{equation*}
\begin{equation*}
   + \int _{k}^{\infty}\Big | \int _{-1}^{0} e^{-ic_2 \omega y}  f_{1n} (y)dy  \Big|^2d\omega + \int _{k}^{\infty}\Big | \int _{-1}^{0} (\alpha +k)e^{-ic_2 \omega y}f_{0n} (y)dy  \Big|^2d\omega.
\end{equation*}
Using integration by parts and the fact that $f_0$ and $f_1$ are vanished at the endpoints, we have
\begin{equation*}
\int _{0}^{1} e^{\pm ic_1 \omega y}  f_{1p} (y)dy = \frac{1}{\pm ic_1 \omega} \int _{0}^{1}e^{\pm ic_1 \omega y}f_{1p} '(y)dy  
\end{equation*}
\begin{equation*}
(\alpha+k)\int _{0}^{1}  e^{\pm ic_1 \omega y}  f_{0p} (y)dy = \frac{(\alpha+k)}{(\pm ic_1 \omega)^2}\int _{0}^{1} e^{\pm ic_1 \omega y}f_{0p}'' (y)dy  
\end{equation*}
\begin{equation*}
\int _{-1}^{0} e^{\pm ic_2 \omega y}  f_{1n} (y)dy = \frac{1}{\pm ic_2 \omega} \int _{-1}^{0}e^{\pm ic_2 \omega y}f_{1n} '(y)dy  
\end{equation*}
\begin{equation*}
(\alpha +k)\int _{-1}^{0}  e^{\pm ic_2 \omega y}  f_{0n} (y)dy = \frac{(\alpha+k)}{(\pm ic_2 \omega)^2}\int _{-1}^{0} e^{\pm ic_2 \omega y}f_{0n}'' (y)dy , 
\end{equation*}
consequently for the first and second terms in \eqref{Line2lemma4.1} we obtain
\begin{equation*}
   \Big | \int _{0}^{1} e^{\pm ic_1 \omega y}  f_{1p} (y)dy  \Big|^2 \leq  \frac{C}{\omega^2}\parallel f_{1p}\parallel^2 _{(1)} (0,1) \leq  \frac{C}{\omega^2} \parallel f_{1p}\parallel^2 _{(1)} (-1,1)
\end{equation*}
and 
\begin{equation*}
    \Big | \int _{0}^{1} (\alpha +k)e^{ic_1 \omega y}f_{0p} (y)dy  \Big|^2 \leq  \frac{C (\alpha ^2+k^2)}{\omega^4}\parallel f_{0p}\parallel^2 _{(2)} (0,1) \leq  \frac{C (\alpha ^2+k^2)}{\omega^4}\parallel f_{0p}\parallel^2 _{(2)} (-1,1),
\end{equation*}
repeating the argument for the other terms in \eqref{Line2lemma4.1} and integrating with respect to $\omega$ the proof is complete.

\end{proof}

\textbf{Remark 2.1.} Obviously, the following inequality holds
\begin{equation*}
\int _{k}^{\infty} \omega ^2 \parallel u(,\omega) \parallel_{(0)}^{2} (\partial \Omega)d\omega \leq k^{-2} \int _{k<|\omega|} \omega ^4  \parallel u(,\omega) \parallel_{(0)}^{2} (\partial \Omega)d\omega \leq
\end{equation*}
\begin{equation*}
 k^{-2} \int _{R} \omega ^4  \parallel u(,\omega) \parallel_{(0)}^{2} (\partial \Omega)d\omega=2\pi k^{-2} \int _{R} \parallel \partial ^{2 }_{t} U(,t) \parallel_{(0)}^{2} (\partial \Omega)dt
\end{equation*}
by the Parseval's identity.\\

Finally, we are ready to establish the  increasing stability estimate of Theorem \ref{maintheorem}\\

{\it Proof of Theorem 1.1.}

\begin{proof}
Without loss of generality, we can assume that $\epsilon <1$ and $3\pi E^{-\frac{1}{4}} <1$, otherwise the bound \eqref{maintheorem} is obvious. Let 
\begin{center}
\begin{equation}
\label{k}
k= \begin{cases}
          K^{\frac{2}{3}}E^{\frac{1}{4}} \quad \text{if} \quad 2^{\frac{1}{4}}K^{\frac{1}{3}}< E ^{\frac{1}{4}} \\
K \hspace{1.19 cm} \text{if}\quad E ^{\frac{1}{4}} \leq 2^{\frac{1}{4}}K^{\frac{1}{3}},
       \end{cases} 
\end{equation}
 \end{center}
if $ E ^{\frac{1}{4}} \leq 2^{\frac{1}{4}}K^{\frac{1}{3}}$, then $k=K$ and 
\begin{equation}
\label{I11}
|I_1 (k)| \leq 2\epsilon ^2.
\end{equation}

If $2^{\frac{1}{4}}K^{\frac{1}{3}}< E ^{\frac{1}{4}}$, 
we can assume that $ E^{-\frac{1}{4}} <\frac{1}{4 \pi}$, otherwise $C<E$ and hence $K<C$ and the bound \eqref{Istability} is straightforward. From \eqref{k}, \eqref{harmonicbound},  Lemma 2.2, \eqref{I1Epsilon} and the equality $\epsilon=e^{-E}$ we obtain
\begin{equation*}
\label{I1epsilon}
|I_1(k) |\leq  CM^{2}  \alpha^2 e^{4\alpha} e^{4k}  e^{\frac{-2E}{\pi}\big( (\frac{k}{K})^4  -1 \big)^{\frac{-1}{2}}} 
\end{equation*}

\begin{equation*}
\leq CM^{2}e^{4\alpha}  \alpha^2 e^{- \frac{2}{\pi}K^{\frac{2}{3}}E^{\frac{1}{2}}(1- \frac{5\pi}{2}   E^{\frac{-1}{4}} )},
\end{equation*}
using the trivial inequality $e^{-t} \leq \frac{6}{t^3}$ for $t>0$ and the assumption at the beginning of the proof, we conclude that 
\begin{equation}
\label{I12}
|I_1(k)|\leq  CM_{0}^{2}  \alpha^2 e^{2\alpha}  \frac{1}{K^2 E ^{\frac{3}{2}}\Big (1-\frac{5 \pi}{2} E^{-\frac{1}{4}}   \Big)^3}.  
\end{equation}
Using \eqref{I1}, \eqref{I11}, \eqref{I12}, and Lemma 4.1 we obtain 
\begin{equation}
\label{lastbound1}
\int ^{+\infty}_{0} \omega ^2 |u(-1,\omega)|^2 d\omega + \int ^{+\infty}_{0} \omega ^2 |u(1,\omega)|^2 d\omega
\end{equation}
\begin{equation*}
    \leq I(k)+  \int_{k}^{\infty} \omega ^2 |u(-1,\omega)|^2 d\omega + \int_{k}^{\infty} \omega ^2 |u(1,\omega)|^2 d\omega
\end{equation*}
\begin{equation*}
\leq 2\epsilon ^2 + C  \Big(  \frac{  (\alpha^2+1)  M^{2}  e^{4\alpha}}{K^2 E ^{\frac{3}{2}}} + \frac{ ( \alpha^2 +1)\parallel f_0 \parallel_{(2)}^2 + \parallel f_1 \parallel_{(1)}^2 }{K^{\frac{2}{3}}E^{\frac{1}{4}}+1} \Big).
\end{equation*}
Using \eqref{lastbound1} and Theorem 3.1, we finally obtain

\begin{equation*}
\parallel f_1 \parallel _{(0)} ^2 (\Omega) +\parallel f_0 \parallel _{(1)} ^2 (\Omega) \leq C e^{\alpha^2}\Big ( \parallel \partial _t U  \parallel _{(0)} ^2 (\partial \Omega \times (0,T))+ \parallel  U \parallel _{(0)} ^2 (\partial \Omega \times (0,T))\Big )
\end{equation*}
\begin{equation*}
\leq  C e^{\alpha^2}\Big ( \parallel \partial _t U  \parallel _{(0)} ^2 (\partial \Omega \times (0,\infty))+ \parallel U \parallel _{(0)} ^2 (\partial \Omega \times (0,\infty))      \Big )
\end{equation*}
\begin{equation*}
\leq C e^{\alpha^2} \Big( \epsilon ^2 +   \frac{ (\alpha^2+1)   M^{2}  e^{8\alpha}}{K^2 E ^{\frac{3}{2}}} + \frac{ ( \alpha^2 +1)\parallel f_0 \parallel_{(2)}^2 + \parallel f_1 \parallel_{(1)}^2 }{K^{\frac{2}{3}}E^{\frac{1}{4}}+1} \Big),
\end{equation*}
due to the Parseval's identity. Since $ K^{\frac{2}{3}} E ^{\frac{1}{4}}<K^2 E ^{\frac{3}{2}}$ for $1<K, 1<E$, the proof is complete.

\end{proof}

\section{Concluding Remarks} In this paper, we studied the inverse source scattering problem with attenuation and many frequencies of the one-dimensional Helmholtz equation in a two-layered medium using multi-frequency
Dirichlet data at the two end points of an interval which contains the compact support of the source. Our results showed a deterioration of stability with growing attenuation/damping constant $\alpha $.\\

Due to the $\alpha$ and term $e^{C\alpha}$, the result of theorem 1.1 is not sufficiently sharp. The quadratic dependence on $\alpha$ in \eqref{Istability} is a consequence of  Carleman estimates for the hyperbolic equation to prove Theorem 2.3. In particular, we used Carleman estimates to trace the dependence of exact observability bounds on the factor $\alpha$. In \cite{IL}, they provided numerical \textcolor{red}{evidence} which \textcolor{red}{agreed} with our result.\\

\section{ Acknowledgment:} 

 This research is supported in part by NSF Award HRD-1824267.

\end{document}